\documentclass{amsart}
\usepackage{amsmath,amscd,amsthm,verbatim,alltt,amsfonts,array}
\usepackage[english]{babel}
\usepackage{latexsym}
\usepackage{amssymb}
\usepackage{euscript}
\usepackage{nicefrac}
\usepackage{graphicx}

\def\opn#1#2{\def#1{\operatorname{#2}}} % to make operators
\opn\chara{char} \opn\length{\ell} \opn\pd{pd} \opn\rk{rk}
\opn\projdim{proj\,dim} \opn\injdim{inj\,dim} \opn\rank{rank}
\opn\depth{depth} \opn\sdepth{sdepth} \opn\fdepth{fdepth}
\opn\grade{grade} \opn\height{height} \opn\embdim{emb\,dim}
\opn\codim{codim}  \opn\min{min} \opn\max{max}

\opn\Tr{Tr} \opn\bigrank{big\,rank}
\opn\superheight{superheight}\opn\lcm{lcm}
\opn\trdeg{tr\,deg}
\opn\reg{reg} \opn\lreg{lreg} \opn\ini{in} \opn\lpd{lpd}
\opn\size{size}

\opn\Spec{Spec} \opn\Supp{Supp} \opn\supp{supp} \opn\Sing{Sing}
\opn\Ass{Ass} \opn\Min{Min}
\opn\Ann{Ann} \opn\Rad{Rad} \opn\Soc{Soc}
\opn\Im{Im} \opn\Ker{Ker} \opn\Coker{Coker} \opn\Am{Am}
\opn\Hom{Hom} \opn\Tor{Tor} \opn\Ext{Ext} \opn\End{End}
\opn\Aut{Aut} \opn\id{id}  \opn\deg{deg}

\opn\nat{nat}
\opn\pff{pf}%   \pf exists already
\opn\Pf{Pf} \opn\GL{GL} \opn\SL{SL} \opn\mod{mod} \opn\ord{ord}
\opn\Gin{Gin} \opn\Hilb{Hilb}
\let\to=\rightarrow

\def\Implies{\ifmmode\Longrightarrow \else
        \unskip${}\Longrightarrow{}$\ignorespaces\fi}
\def\implies{\ifmmode\Rightarrow \else
        \unskip${}\Rightarrow{}$\ignorespaces\fi}
\def\iff{\ifmmode\Longleftrightarrow \else
        \unskip${}\Longleftrightarrow{}$\ignorespaces\fi}

\let\:=\colon
\newtheorem{theorem}{Theorem}
\theoremstyle{plain}

\newtheorem{conjecture}{Conjecture}

\newtheorem{definition}{Definition}
\newtheorem{example}{Example}

\newtheorem{proposition}{Proposition}
\newtheorem{remark}{Remark}

\numberwithin{equation}{section}

\begin{document}
\title{\bf Stanley depth of monomial ideals  }

\author{ Dorin Popescu}

\thanks{The  support from  grant    PN-II-RU-TE-2012-3-0161  of Romanian Ministry of Education, Research and Innovation are gratefully acknowledged.}

\address{Dorin Popescu,  Simion Stoilow Institute of Mathematics of Romanian Academy, Research unit 5,
 P.O.Box 1-764, Bucharest 014700, Romania}
\email{dorin.popescu@imar.ro}

\begin{abstract} Let $I\supsetneq J$ be  two  monomial ideals of a polynomial algebra over a field generated in degree $\geq d$, resp. $\geq d+1$ . We study when the Stanley Conjecture holds for $I/J$ using the recent result of \cite{IKM} concerning the polarization.

\end{abstract}
\maketitle

\section*{Introduction}
Let $K$ be a field and $S=K[x_1,\ldots,x_n]$ be the polynomial $K$-algebra in $n$ variables. Let  $I\supsetneq J$ be  two  monomial ideals of $S$ and suppose that  $I$ is generated by  monomials of degrees $\geq d$   for some positive integer $d$.  Using a multigraded isomorphism we may assume either that $J=0$, or $J$ is generated in degrees $\geq d+1$.

If $I,J$ are squarefree monomial ideals then $d$ is a lower bound of $\depth_SI/J$ by \cite[Proposition 3.1]{HVZ} (see also \cite[Lemma 1.1]{P}).   Proposition \ref{lb} gives a lower bound of $\depth_SI/J$ in terms of degrees also in the case when $I,J$ are not squarefree using the polarization and the so called the canonical form of $I/J$ (see \cite{A}).

 A Stanley decomposition of  a multigraded $S$-module $M$ is a finite family $\mathcal{D} = (S_l, u_l)_{l \in L}$ in which $u_l$ are homogeneous elements of $M$ and $S_l$ are multigraded $K-$algebra retract of $S$ for all $l  \in L$ such that $S_l \cap \Ann_Su_l = 0$ and $M = \sum_{l \in L} S_lu_l$ as a multigraded $K-$vector space.  The Stanley depth of $\mathcal{D}$, denoted by $\sdepth (\mathcal{D})$, is the depth of the $S-$module $\sum_{l \in L} S_lu_l$. The Stanley depth of $M$ is defined as   $$ \sdepth\ (M) :=\max\{\sdepth \ ({\mathcal D})\ |\  {\mathcal D}\; \text{is a
Stanley decomposition of}\;  M \}.$$

Depth and Stanley depth behave in a different way for instance $\depth_S(M\oplus M')=\min\{\depth_SM,\depth_SM'\}$ while for sdepth it can happen  $\sdepth_S(M\oplus M')>\min\{\sdepth_SM,\sdepth_SM'\}$ sometimes
as seen in \cite[Examples 14, 16]{IZ} with the help of \cite{A1}. These results were obtained using the so called the Hilbert depth  (see \cite{BKU}, \cite{U}). The same notion is important also in other properties of depth and Stanley depth (see \cite[Proposition 2.4]{Sh}).

An upper bound for $\depth_SM$ could be given by the following  conjecture.

\begin{conjecture}(Stanley \cite{S}) \label{cm} $\depth_SM\leq \sdepth_SM$.
\end{conjecture}

It will be very nice if this conjecture holds for $M=I/J$.
Recently Ichim, Katth\"an and  Moyano-Fern\'andez proved that Stanley's Conjecture holds for all factors $I/J$ as above if and only if it holds for their polarizations \cite[Theorem 4.3]{IKM}. Thus we may restrict to the case when $I,J$ are squarefree monomial ideals. Unfortunately, there are few results in this case in spite of the many papers appeared on this subject (see \cite{HVZ}, \cite{P2}, \cite{P1}, \cite{Is}, \cite{P0}, \cite{Ci}, \cite{P}). It is the purpose of our paper to study what   these few  results say in the non squarefree case using  \cite[Theorem 4.3]{IKM}. We use here the lower bound given by Proposition \ref{lb} (see  Theorems \ref{main}, \ref{main1} and Proposition \ref{main2}).

A particular case of this conjecture is the following one.

\begin{conjecture} \label{c}   Suppose that $I \subset S$ is minimally generated by some squarefree monomials $f_1,\ldots,f_r$ of degrees $d$,  and a  set $E$  of squarefree monomials of degree $\geq d+1$. If  $\sdepth_S I/J=d+1$ then $\depth_S I/J \leq d+1$.
\end{conjecture}
This conjecture is studied in \cite{PZ}, \cite{PZ1}, \cite{PZ2}, \cite{AP}, \cite{P3}  when $r\leq 4$ and some cases when $r=5$ (see Theorems \ref{m1}, \ref{m2}).  Proposition \ref{pr} proves  Conjecture \ref{c} also when $r=6$ but $d=1$ and $E=\emptyset$.

\section{A lower bound of depth and Stanley depth}

Let $S=K[x_1,\ldots,x_{n-1}]$ be the polynomial $K$-algebra over a field $K$ and  $J \subsetneq I\subset R$ two monomial ideals. Denote by $G(I)$, respectively $G(J)$, the minimal monomial system of generators of $I$, respectively $J$.

  A very important result concerning the Stanley depth is given by \cite[Corollary 4.4]{IKM} and we recall it below.

\begin{theorem} (Ichim, Katth\"an,  Moyano-Fern\'andez) \label{ikm} Let $J \subsetneq  I$ be monomial ideals of $S$, and let $I^p\subset J^p$ be their  (complete)
polarizations. Then
$$\sdepth_S I/J -\depth_S I/J = \sdepth I^p/J^p-\depth I^p/J^p.$$
\end{theorem}
For $i\in [n]$ let $e_i=\max_{u\in G(I)\cup G(J)} \deg_{x_i}u$ and set $e_{I/J}=\sum_{i\in [n],e_i>0} (e_i-1)$. We have $e_{I/J}=\depth I^p/J^p-\depth_SI/J$, that is $e_{I/J}$ is the number of the new variables necessary for polarization. Suppose that $I$ is generated by some monomials $f_1,\ldots,f_r$ of degrees $d_{I/J}$ and a set of monomials $E$ of degrees  $\geq d_{I/J}+1$. Then
\begin{proposition}\label{ds} $\depth_SI/J\geq d_{I/J}-e_{I/J}$ and $\sdepth_SI/J\geq d_{I/J}-e_{I/J}$.
\end{proposition}
\begin{proof} By \cite[Proposition 3.1]{HVZ} (see also \cite[Lemma 1.1]{P}) we have $\depth I^p/J^p\geq d_{I/J}$ because $I^p,J^p$ are squarefree monomial ideals. Note that by polarization the degrees of monomials are preserved. It follows that $\depth_SI/J=\depth I^p/J^p-e_{I/J}\geq d_{I/J}-e_{I/J}$. The inequality concerning sdepth is similar but easier since obviously the sdepth is $\geq d_{I/J}$ in the case of a factor of some squarefree monomial ideals.
\end{proof}

\begin{example}\label{ex}{\em Let $n=3$, $d=12$, $I=(x_1^3x_2^4x_3^5,x_1^{10}x_2^2)$. Note that $e_1=10$, $e_2=4$, $e_3=5$ and $e_I=16$.}
\end{example}
\begin{remark}\label{r} {\em In the above example  Proposition \ref{ds} gives $\depth_SI\geq -4$ which is obvious. This situation will  be next improved  considering the so called the canonical form of $I$ given by \cite{A}.}
\end{remark}

We recall some definitions and  results  from \cite{A}.
\begin{definition} \label{can}
 {\em
 The power $x_n^r$ {\em enters in a monomial} $u$ if  $x_n^r|u$ but $x_n^{r+1}\nmid u$.
We say that $I/J$ is {\em of type} $(k_1,\ldots,k_s)$ {\em with respect to} $x_n$ if $x_n^{k_i}$ are all the powers of $x_n$ which enter in a monomial of $G(I)\cup G(J)$ for $i\in [s]$ and $1\leq k_1<\ldots<k_s$.
$I/J$ is {\em in the canonical form with respect to} $x_n$ if $I/J$ is of type $(1,\ldots,s)$ for some $s\in {\mathbb  N}$ and
we  say that $I/J$ is {\em the canonical form} if it is in the canonical form with respect to all variables $x_1, \ldots, x_n$.}
\end{definition}

\begin{remark} \label{canf}{\em
Suppose that  $I/J$ is  of type $(k_1,\ldots,k_s)$ with respect to $x_n$. It is easy to get the {\em canonical form} $I'/J'$ of $I/J$ {\em with respect to} $x_n$ replacing $x_n^{k_i}$ by $x_n^i$ whenever $x_n^{k_i}$ enters in a generators of $G(I)\cup G(J)$. Applying by recurrence this  procedure for other variables we get the {\em canonical form} of $I/J$, that is with respect to all variables.  }
\end{remark}

\begin{theorem} (A. Popescu \cite[Theorems 1, 2]{A}) \label{a} Let $I'/J'$ be the canonical form of $I/J$. Then $\sdepth_SI'/J'=\sdepth_SI/J$
and  $\depth_SI'/J'=\depth_SI/J$.
\end{theorem}
\begin{definition} \label{index} {\em Let $I'/J'$ be the canonical form of $I/J$ and set $t_{I/J}=\max\{d_{I'/J'}-e_{I'/J'},0\}$ (we may have $d_{I'/J'}<e_{I'/J'}$ as shows Example \ref{exe}). We call $t_{I/J}$ the {\em index of } $I/J$.  When $J=0$ we write $t_I$ instead $t_{I/J}$ for simplicity. If $I,J$ are squarefree monomial ideals then $t_{I/J}=d_{I/J}$. }
\end{definition}

Using the terminology defined above we get a better lower bound for sdepth and depth as in Proposition \ref{ds}.

\begin{proposition} \label{lb} $\depth_SI/J\geq t_{I/J}$ and  $\sdepth_SI/J\geq t_{I/J}$.
\end{proposition}
\begin{proof}  By Theorem \ref{a}  we have $\depth_SI/J=\depth_SI'/J'\geq \max\{d_{I'/J'}-e_{I'/J'},0\}=t_{I/J}$. The second inequality holds similarly. \hfill\
\end{proof}

\begin{remark} {\em This lower bound is  easy to describe but it is not the best known lower bound. For example, when $J=0$ then a better lower bound is given by
$1+\size(I)$ in the terminology of \cite{L}, \cite{HPV}. More precisely, if $n=3$, $d_I=1$,  $I=(x_1,x_2x_3)=(x_1,x_2)\cap (x_1,x_3)$ then
$\size(I)=1$ and  $t_I=d_I$ since $I$ is squarefree. Thus $1+\size(I)>t_I$.
}
\end{remark}
\begin{remark}{\em In Example \ref{ex} note that $I$ is of type $(3,10)$ with respect to $x_1$, of type $(2,4)$ with respect to $x_2$ and of type $(5)$ with respect to $x_3$. Then the canonical form of $I$ is $I'=(x_1x_2^2x_3,x_1^2x_2)$. Note that $I$ is generated by monomials of degrees $12$ but in $I'$ one generator has degree $4$ and the other $3$. Clearly,  $e_{I'}=2$, $d_{I'}=3$ and so the index $t_I$ of $I$ is $1$. Thus  Proposition \ref{lb} says that $\depth_SI\geq 1$, which is also trivial but anyway better than what follows from Proposition \ref{ds}  (see Remark \ref{r}).
}
\end{remark}
\section{Stanley depth of monomial ideals which are not necessarily squarefree}

Suppose that $I$ is minimally generated by some squarefree monomials $f_1,\ldots,f_r$ of degrees $ d$  for some $d\in {\mathbb N}$ and a set of squarefree monomials $E$ of degree $\geq d+1$.
 Let  $B$  be the set of the squarefree monomials of degrees $d+1$  of $I\setminus J$.

We start recalling two results of  \cite{P3} (see also \cite{PZ2} and \cite{AP}).

\begin{theorem} \label{m1} Conjecture \ref{c} holds for $r\leq 4$.
\end{theorem}
\begin{theorem} \label{m2} Conjecture \ref{c}  holds for $r=5$ if
there exists  $j\not\in \cup_{i\in [5]} \supp f_i$, $j\in [n]$ such that $(B\setminus E) \cap (x_j)\not =\emptyset$ and $E\subset (x_j)$.
\end{theorem}

For simplicity we denote $t=t_{I/J}$, that is the index of $I/J$.
\begin{theorem}\label{main}  Let $J \subsetneq  I$ be monomial ideals of $S$ not necessarily squarefree and $I'/J'$ the canonical form of $I/J$. Suppose that $I'$ is generated by $r'$ monomials $f_1,\ldots,f_{r'}$ of degree $d_{I'/J'}$ and a set $E'$ of monomials of degree $\geq d_{I'/J'}+1$. Let $B'$ be the set of monomials of degree $d_{I'/J'}+1$ from $I'\setminus J'$. Assume that $\sdepth_SI/J=t+1$. Then the following statements hold:
\begin{enumerate}
\item{} If   $r'\leq 4$ then  $\depth_SI/J\leq t+1$,
\item{} If $r'=5$ and there exists  $j\not\in \cup_{i\in [5]} \supp f_i$, $j\in [n]$ such that $(B'\setminus E') \cap (x_j)\not =\emptyset$ and $E'\subset (x_j)$,  then  $\depth_SI/J\leq t+1$.
\end{enumerate}
\end{theorem}
\begin{proof} Let $I'/J'$ be the canonical form of $I/J$.  By Theorem \ref{ikm} we have
$$\sdepth I'^p/J'^p=\sdepth_S I'/J'-\depth I'^p/J'^p+\depth_SI'/J'=d_{I'/J'}+1.$$
Since $I'^p$ is generated by $r'$ squarefree monomials of degree $d_{I'/J'}$ and a set $E'^p$ of squarefree  monomials of degree $d_{I'/J'}+1$ we get by Theorem \ref{m1} that $\depth I'^p/J'^p\leq d_{I'/J'}+1$. Hence $\depth_SI/J=\depth_SI'/J'\leq t+1$ by Theorem \ref{a}, that is (1) holds. The proof of (2) is the same using Theorem \ref{m2} instead Theorem \ref{m1}.
\end{proof}
\begin{remark} {\em Let $I$ be generated by some monomials $h_1,\ldots,h_r$ of degree $d$ and a set of monomials $E$ of monomials of degree $\geq d+1$. It is possible that $I'$ is generated by $f_1,\ldots,f_{r'}$ of degrees $d_{I'/J'}$ with $r'>r$ and a set $E'$ of monomials of degree $\geq d_{I'/J'}+1$. For example when $n=2$ and $I=(x_1^3x_2^4,x_1^{11}x_2)$ we have $r=1$ and we see that $I'=(x_1x_2^2,x_1^2x_2)$ has $r'=2$.}
\end{remark}

\begin{example} {\em Let $n=2$, $d=1$, $I=(x_1)$, $J=(x_1x_2^2)$. Then $e_1=1$, $e_2=2$, $e_{I/J}=1$, $t=0$ and $I^p/J^p=(x_1)/(x_1x_2y_2)$, where $y_2$ is the new variable from polarization. We have $I/J\cong x_1K[x_1]\oplus x_1x_2K[x_1]$ as graded $K$-vector spaces.
 Thus $\sdepth_SI/J=1=t+1$. By (1) of the above theorem
we get $\depth_SI/J\leq 1$, the inequality being in fact an equality.}
\end{example}

\begin{theorem}\label{main1}  Let $J \subsetneq  I$ be monomial ideals of $S$ not necessarily squarefree.  Assume that $\sdepth_SI/J=t$. Then $\depth_SI/J=t$
\end{theorem}
The proof is similar to the proof of Theorem \ref{main} using now \cite[Theorem 4.3]{P} instead Theorem \ref{m1}.

\begin{example}\label{exe}  {\em Let $n=2$, $d=1$, $I=(x_2)$, $J=(x_1^2x_2,x_1x_2^2)$. Then $e_1=e_2=e_{I/J}=2$, $t=\max\{-1,0\}=0$ and $I^p/J^p=(x_2)/(x_1y_1x_2,x_1x_2y_2)$, where $y_1,y_2$ are the new variables from polarization.  Since $x_2$ induces a nonzero element of the socle of $I/J$ we see that $\sdepth_SI/J=0$.
 Thus $\sdepth_SI/J=t=0$. By  the above theorem
we get $\depth_SI/J= 0$.}
\end{example}

\section{ Stanley depth of a factor of squarefree monomial ideals}

The above theorem implies the following corollary.
\begin{proposition} \label{pr} Suppose that $I \subset S$ is minimally generated by  $6$ variables\\ $\{x_1,\ldots,x_6\}$ and  $J \subsetneq  I$ a squarefree  monomial ideal. If  $\sdepth_SI/J=2$ then  $\depth_SI/J\leq 2$.
\end{proposition}
\begin{proof} By \cite[Proposition 1.3]{PZ} we see that there exists $c=x_6x_kx_q\not\in J$ for $6<k<q\leq n$.  Let $B$ be the set of all squarefree monomials from $I\setminus J$ and $\tilde I$ be the ideal generated by $x_1,\ldots, x_5$ and ${\tilde E}=B\setminus ((x_1,\ldots,x_5)\cup [x_6,c])$. Set  ${\tilde J}=J\cap {\tilde I}$. Then for $j=6$ we have ${\tilde E}\subset (x_j)$. In the following exact sequence
$$0\to {\tilde I}/{\tilde J}\to I/J\to I/J+{\tilde I}\to 0$$
the last term is isomorphic with $(x_6)/(x_6)\cap (J+{\tilde I})$ and has depth $\geq 2$ and sdepth $3$ because  it has just the interval $[x_6,c]$. Suppose that $\sdepth_SI/J=2$. By \cite[Lemma 2.2]{R} we get $\sdepth_S{\tilde I}/{\tilde J}\leq 2$. When $\sdepth_S{\tilde I}/{\tilde J}= 1$ then it is enough to apply \cite[Theorem 4.3]{P}. If $\sdepth_S{\tilde I}/{\tilde J}= 2$ and $(B\setminus {\tilde E})\cap (x_j)\not=\emptyset$ then it is enough to apply Theorem \ref{m2}.

Now suppose that $(B\setminus {\tilde E})\cap (x_j)=\emptyset$, that is $B\cap (x_6)\cap (x_1,\ldots,x_5)=\emptyset$.
In the following exact sequence
$$0\to (x_6)/(x_6)\cap J\to I/J\to I/(J,x_6)\to 0$$
if the last term has sdepth $\geq 3$ then the first term has sdepth $\leq 2$ as above and so also depth $\leq 2$. Otherwise, the last term has sdepth $\leq 2$. But the last term is isomorphic with $(x_1,\ldots,x_5)/  (x_1,\ldots,x_5)\cap J$ because  $B\cap (x_6)\cap (x_1,\ldots,x_5)=\emptyset$. Thus in the exact sequence
$$0\to (x_1,\ldots,x_5)/  (x_1,\ldots,x_5)\cap J\to I/J\to I/(J,x_1,\ldots,x_5)\to 0$$
the first term has sdepth $\leq 2$ and so its depth $\leq 2$ by Theorem \ref{m2} when there exists $k>6$ such that $B\cap (x_1,\ldots,x_5)\cap (x_k)\not =\emptyset$. Otherwise, $J\geq (x_1,\ldots,x_5) (x_6,\ldots,x_n)$ and we get $$\depth_S(x_1,\ldots,x_5)/(x_1,\ldots,x_5)\cap J=\depth_{\tilde S}(x_1,\ldots,x_5){\tilde S}/(x_1,\ldots,x_5)\cap J\cap {\tilde S}\leq 1$$
 for ${\tilde S}=K[x_1,\ldots,x_5]$. Since the last term is isomorphic with $(x_6)/J\cap (x_6)$ it has depth $\geq 2$ and the Depth Lemma ends the proof.
\hfill\  \end{proof}

\begin{proposition}\label{main2} Suppose that $I \subset S$ is minimally generated by  $6$ variables \\
$\{x_1,\ldots,x_6\}$ and   $J \subsetneq  I$ is a monomial ideal  not necessarily squarefree. Suppose that  $\sdepth_SI/J=t+1$.
 Then  $\depth_SI/J\leq t+1$.
\end{proposition}
The proof is similar to the proof of Theorem \ref{main} using now Proposition \ref{pr} instead Theorem \ref{m1}.
\begin{example}  {\em
Let $n=7$,  $I=(x_1,\ldots,x_6),$
$ J=(x_1^2,x_1x_2,\ldots,x_1x_5,x_1x_7)$. Then $t=0$.
The element ${\hat x}_1\in I/J$ induced by $x_1$ is annihilated by all variables but $x_6$.
It follows that $\sdepth_SI/J\leq 1$. Thus $\sdepth_SI/J\leq t+1$ and so $\depth_SI/J\leq 1$ by Proposition \ref{main2}. Note that
$I^p/J^p=(x_1,\ldots,x_6)/(x_1y,x_1x_2,\ldots,x_1x_5,x_1x_7)$ has sdepth $\leq 2$ because now the element of $I^p/J^p$ induced by $x_1$ is annihilated by all variables but $x_6,y$.}
\end{example}

 \end{document}